\documentclass[10pt]{amsart}

\usepackage{enumerate}
\usepackage{amsmath,amssymb}

\newtheorem{theorem}{Theorem}[section]

\newtheorem{lemma}[theorem]{Lemma}

\newtheorem{proposition}[theorem]{Proposition}

\theoremstyle{definition}

\newtheorem{example}[theorem]{Example}

\def\B{{\mathcal{B}}}

\def\F{{\mathcal{F}}}
\def\H{{\mathcal{H}}}
\def\P{{\mathcal{P}}}
\def\Q{{\mathcal{Q}}}
\def\U{{\mathcal{U}}}

\def\C{{\mathbb{C}}}
\def\N{{\mathbb{N}}}

\usepackage{comment}

\begin{document}

\title[Isometries between projection lattices]{Isometries between projection lattices \\
of von Neumann algebras}

\author[M. Mori]{Michiya Mori}

\address{Graduate School of Mathematical Sciences, the University of Tokyo, Komaba, Tokyo, 153-8914, Japan.}
\email{mmori@ms.u-tokyo.ac.jp}

%\thanks{}
%46L40,


\date{}

\begin{abstract}
We investigate surjective isometries between projection lattices of two von Neumann algebras. 
We show that such a mapping is characterized by means of Jordan $^*$-isomorphisms. 
In particular, we prove that two von Neumann algebras without type I$_1$ direct summands are Jordan $^*$-isomorphic if and only if their projection lattices are isometric. 
Our theorem extends the recent result for type I factors by G.P. Geh\'er and P. \v{S}emrl, which is a generalization of Wigner's theorem.
\end{abstract}
\maketitle
\thispagestyle{empty}

\section{Introduction}
The study of isometries between operator algebras has a long history. 
The first achievement in this field dates back to 1951 by Kadison \cite{K}. 
He proved that if $\phi\colon A\to B$ is a complex linear surjective isometry between two unital C$^*$-algebras, then $\phi(1)$ is a unitary operator in $B$ and the mapping $x\mapsto \phi(1)^{-1}\phi(x)$, $x\in A$ is a Jordan $^*$-isomorphism. 
(A linear bijection $J\colon A\to B$ between two C$^*$-algebras is called a \emph{Jordan $^*$-isomorphism} if it satisfies $J(x^*)=J(x)^*$ and $J(xy+yx)=J(x)J(y)+J(y)J(x)$ for any $x, y\in A$.) 
On the other hand, recall that the celebrated Mazur-Ulam theorem asserts that every surjective isometry between two Banach spaces is affine. 
Also, Mankiewicz's generalization \cite{Ma} of this theorem states that every surjective isometry between open connected subsets of Banach spaces is affine. 
This gives rise to a question which asks whether an analogous result holds for isometries between substructures of operator algebras. 
In recent years, there have been several great developments in such a study. 
Hatori and Moln\'ar proved that every surjective isometry between unitary groups of two von Neumann algebras extends uniquely to a real linear surjective isometry \cite{HM}. 
Tanaka applied this theorem to consider Tingley's problem for finite von Neumann algebras \cite{Tan}. 
Tingley's problem asks whether every surjective isometry between unit spheres of two Banach spaces admits a real linear extension.  
Stimulated by Tanaka's research, Tingley's problem began to be considered in various settings of operator algebras. 
See \cite{Mor}, \cite{Pe} and \cite{MO} for latest progresses in such a study.\smallskip

In this paper, we consider surjective isometries between projection lattices of two von Neumann algebras.  
Since projection lattices play very important roles in the theory of von Neumann algebra, 
it is natural to ask whether a result similar to Hatori and Moln\'ar's theorem holds for isometries between projection lattices. 
Here we give an observation which seems to imply an affirmative answer to this question. 
Let $M$ be a von Neumann algebra. 
The symbol $\P(M)$ denotes the projection lattice of $M$, 
and the symbol $\U(M)$ means the unitary group of $M$. 
That is, $\P(M):=\{p\in M\mid p=p^2=p^*\}$ and $\U(M):=\{u\in M\mid u^*u=1=uu^*\}$. 
Consider two projections $p_1:=\operatorname{diag}(1,0),\, p_2:=\operatorname{diag}(0,1) \in \P(M_2(M))$. 
Then we have 
\[
\left\{p\in\P(M_2(M)) \middle| \lVert p-p_1\rVert = \frac{1}{\sqrt{2}} = \lVert p-p_2\rVert\right\} 
= 
\left\{\frac{1}{2}
\begin{pmatrix}
1 & u\\
u^* & 1
\end{pmatrix}
\middle| u\in \U(M)
\right\}.
\]
This set is isometric to $\U(M)/2 = \{u/2\mid u\in\U(M)\}$. 
By the Hatori-Moln\'ar theorem, this set contains much information about $M$. 

It is well-known that the distance between two distinct connected components in the projection lattice of a von Neumann algebra is always $1$. 
Thus, in order to consider surjective isometries between projection lattices of von Neumann algebras, it suffices to consider isometries between connected components. 
In this paper, a connected component in $\P(M)$ which contains more than one element
is called a \emph{Grassmann space} in $M$. 
We know that every Jordan $^*$-isomorphism between two von Neumann algebras restricts to isometries between Grassmann spaces. 
Another example of an isometry between Grassmann spaces on $M$ can be obtained by the mapping $p\mapsto p^{\perp}\, (:= 1-p)$. 
In this paper, we show that every surjective isometry between Grassmann spaces can be decomposed to such mappings (Theorem \ref{thm-G}).\smallskip

As for the case $M=\B(\H)$, the research of isometries between Grassmann spaces is motivated by Wigner's unitary-antiunitary theorem.  
Wigner's theorem plays an important role in the mathematical foundation of quantum mechanics. 
Let $\P_1(\H)$ stand for the collection of rank $1$ projections on a complex Hilbert space $\H$. 
Note that $\P_1(\H)$ is a Grassmann space in $\B(\H)$. 
Wigner's theorem shows that every surjective isometry from $\P_1(\H)$ onto itself extends to a Jordan $^*$-automorphism on $\B(\H)$. 
See Introduction of \cite{BJM}.  
After several attempts (e.g.\ \cite{BJM}, \cite{GS1}) to generalize this result, 
Geh\'er and \v{S}emrl recently gave complete descriptions of surjective isometries between two Grassmann spaces in $\B(\H)$ \cite{GS2}. 
They made use of the idea of geodesics between two projections, which is also essential in our proof of Theorem \ref{thm-G}. 
See also \cite{S}, \cite{G}, \cite{Pa} and \cite{Mol}, \cite{QWWY}, in which mappings between projection lattices with an assumption which is different from ours are studied.\smallskip

In Section \ref{Grassmann}, we give the proof of Theorem \ref{thm-G}. 
Throughout the proof, we depend on the idea by Geh\'er and \v{S}emrl for $\B(\H)$ in \cite{GS2}, but we need more discussions in order to consider general von Neumann algebras.
Our strategy is as follows. 
We see that we may assume every projection in the Grassmann spaces is finite or properly infinite, and the mapping preserves orthogonality in both directions.   
By the Hatori-Moln\'ar theorem combined with the idea about $M_2(M)$ as above,  we can construct a Jordan $^*$-isomorphism between small subspaces. 
Using that, we extend the given mapping to a bijection between whole projection lattices which preserves orthogonality in both directions. 
Finally, we make use of a theorem due to Dye \cite{D} to complete the proof. 

In Section \ref{lattice}, by means of Theorem \ref{thm-G}, we consider surjective isometries between projection lattices. 
We show that two von Neumann algebras without type I$_1$ direct summands
are Jordan $^*$-isomorphic if and only if their projection lattices are isometric (Theorem \ref{thm-l}). 
We also consider concrete cases when two von Neumann algebras are factors.

\section{Isometries between Grassmann spaces}\label{Grassmann}
In this section, we use standard terminology and basic properties concerning geometry of projection lattices. 
See for example \cite[Chapter 6]{KR} or \cite[Chapter V.1]{Tak}. 

Let $M$ be a von Neumann algebra. 
The symbol $z(p)$ denotes the central support of $p$ for a projection $p \in \P(M)$.
Let $\P$ be a Grassmann space in $M$ and $p\in \P$. 
It is an elementary exercise to show that a projection $q\in \P(M)$ belongs to $\P$ if and only if $p$ is unitarily equivalent to $q$ in $M$.  
Thus the pair $(z(p), z(p^{\perp}))$ of central projections does not depend on the choice of $p\in\P$. 
In this paper, a Grassmann space $\P$ in $M$ is said to be \emph{proper} if 
$z(p)=1=z(p^{\perp})$ for every $p\in\P$. 
Fix a projection $p_0\in \P$.
The mapping $p \mapsto pz(p_0)z(p_0^{\perp})$ determines a bijection from $\P$ onto a proper Grassmann space in the von Neumann algebra $Mz(p_0)z(p_0^{\perp})$. 
Therefore, in order to consider surjective isometries between Grassmann spaces, 
we may assume that these Grassmann spaces are proper.\smallskip

The main theorem of this section is the following one:
\begin{theorem}\label{thm-G}
Let $M, N$ be von Neumann algebras and $\P\subset M$, $\Q\subset N$ be proper Grassmann spaces. 
Suppose $T\colon \P\to\Q$ is a surjective isometry.
Then there exist a Jordan $^*$-isomorphism $J\colon M\to N$ and a central projection $r\in \P(N)$ which satisfy
\[
T(p) = J(p)r + J(p^{\perp})r^{\perp},\quad p\in\P.
\]
\end{theorem}

We construct this section to some extent along the lines of the paper \cite{GS2} by Geh\'er and \v{S}emrl.
First we give a generalization of Halmos' two projections theorem in the setting of von Neumann algebras. 
The following result is known among operator algebraists (see for example \cite[pp.\ 306--308]{Tak}), but we concisely give its proof here. 

\begin{lemma}
Let $p, q$ be projections in a von Neumann algebra $M\subset \B(\H)$. 
Then there exist a partial isometry $v\in M$ and positive elements $a, b\in M$ which satisfy the following conditions: 
\[
\begin{gathered}
vv^*= p- p\wedge q - p\wedge q^{\perp},\quad v^*v= p^{\perp} - p^{\perp}\wedge q - p^{\perp}\wedge q^{\perp}, \\
a^2+b^2 = p- p\wedge q - p\wedge q^{\perp}, \\
q - p\wedge q - p^{\perp}\wedge q = a^2 + abv + v^*ba + v^*b^2v.
\end{gathered}
\]
\end{lemma}

\begin{proof}
Put $e_1:= p- p\wedge q - p\wedge q^{\perp}$ and $e_2:= p^{\perp} - p^{\perp}\wedge q - p^{\perp}\wedge q^{\perp}$. 
It follows that the range projection of $x:= e_1(q-p\wedge q-p^{\perp}\wedge q)e_2$ is $e_1$, 
and the range projection of $x^*$ is $e_2$. 
Let $x= v\lvert x \rvert = \lvert x^* \rvert v$ be the polar decomposition. 
Then $vv^*= e_1$ and $v^*v= e_2$. 
We can identify each $y\in (e_1+e_2)M(e_1+e_2)$ with 
$\begin{pmatrix}
e_1ye_1 & e_1yv^* \\
vye_1 & vyv^*
\end{pmatrix} \in M_2(e_1Me_1)$. 
Then $q-p\wedge q-p^{\perp}\wedge q\, (\leq e_1+e_2)$ is identified with 
\[
\begin{split}
&\begin{pmatrix}
e_1(q-p\wedge q-p^{\perp}\wedge q)e_1 & e_1(q-p\wedge q-p^{\perp}\wedge q)v^* \\
v(q-p\wedge q-p^{\perp}\wedge q)e_1 & v(q-p\wedge q-p^{\perp}\wedge q)v^*
\end{pmatrix} \\
=& \begin{pmatrix}
e_1(q-p\wedge q-p^{\perp}\wedge q)e_1 & \lvert x^* \rvert \\
\lvert x^* \rvert & v(q-p\wedge q-p^{\perp}\wedge q)v^*
\end{pmatrix}
\in M_2(e_1Me_1).
\end{split}
\] 
Put $a:= (e_1(q-p\wedge q-p^{\perp}\wedge q)e_1)^{1/2}$ and $b:= (v(q-p\wedge q-p^{\perp}\wedge q)v^*)^{1/2}$. 
Since $\begin{pmatrix}
a^2 & \lvert x^* \rvert \\
\lvert x^* \rvert & b^2
\end{pmatrix}$ is a projection, it follows that $a, b$ and $\lvert x^* \rvert$ commute, $a^2+b^2= e_1$ and $\lvert x^* \rvert = ab$. 
Thus $q-p\wedge q-p^{\perp}\wedge q$ corresponds to $\begin{pmatrix}
a^2 & ab \\
ba & b^2
\end{pmatrix}$, and hence 
$q-p\wedge q-p^{\perp}\wedge q = a^2+ abv + v^*ba + v^*b^2v$.
\end{proof}

This lemma generalizes Halmos' two projection theorem. (See for example \cite[Theorem 2.1]{GS2}.)  
Consider the decomposition 
\[
\H = (p\wedge q^{\perp}) \H \oplus (p^{\perp} \wedge q) \H \oplus (p\wedge q) \H \oplus (p^{\perp}\wedge q^{\perp}) \H \oplus e_1\H \oplus e_2\H. 
\]
By the partial isometry $v\in M$ in the above lemma, we can identify $(e_1+e_2)M(e_1+e_2)$ with $M_2(e_1Me_1)$, and we can realize $p$ and $q$ as the following:
\[
p= \begin{pmatrix}
1&0&0&0&0&0 \\
0&0&0&0&0&0 \\
0&0&1&0&0&0 \\
0&0&0&0&0&0 \\
0&0&0&0&1&0 \\
0&0&0&0&0&0
\end{pmatrix},\quad 
q= \begin{pmatrix}
0&0&0&0&0&0 \\
0&1&0&0&0&0 \\
0&0&1&0&0&0 \\
0&0&0&0&0&0 \\
0&0&0&0&a^2&ab \\
0&0&0&0&ba&b^2
\end{pmatrix}, 
\]
and $a,b$ are injective operators in $(e_1Me_1)_+$ which satisfy $a^2+b^2=e_1$.\smallskip

For two projections $p, q \in \P(M)$, we write $p\vartriangle q$ 
if there exists a central projection $r \in M$ such that 
$pr \perp qr$ and $p^{\perp}r^{\perp} \perp q^{\perp}r^{\perp}$. 
Note that this relation is a generalization of the relation which is written as ``$\sim$'' in the paper \cite{GS2}. 
In our paper, we save the symbol $\sim$ for the Murray-von Neumann equivalence.
i.e.\ We write $p\sim q$ when there exists a partial isometry $v\in M$ such that $vv^*=p$ and $v^*v=q$. 
In addition, we write $p\prec q$ when there exists a partial isometry $v\in M$ such that $vv^*=p$ and $v^*v\leq q$.

\begin{proposition}
Let $M \subset \B(\H)$ be a von Neumann algebra, $\P$ be a Grassmann space in $M$ and $p, q\in \P$ with $\lVert p-q \rVert = 1$.  
Then we have $p\vartriangle q$ if and only if the following holds.\smallskip

\noindent\underline{Condition}\quad 
Set $m(p,q):= \{e \in \P \mid \lVert e-p\rVert = \lVert e-q\rVert = 1/{\sqrt{2}}\}$. 
Then $m(p,q)$ is not empty, and for every $p_0 \in m(p,q)$, there exists a unique path $\gamma\colon [0, \pi/2] \to \P$ which satisfies 
\[
\gamma(0)=p,\quad \gamma(\pi/2)=q,\quad \gamma(\pi/4)=p_0
\] 
and
\[
\lVert\gamma(\theta_1)-\gamma(\theta_2)\rVert = \sin \lvert \theta_1-\theta_2\rvert
\]
for all $\theta_1, \theta_2 \in [0, \pi/2]$.
\end{proposition}
\begin{proof}
The discussion in the paper \cite{GS2} can be applied almost verbatim, so we give only a sketch of the proof.
 
Suppose $p\vartriangle q$. 
It suffices to consider the case $p \perp q$. 
Fix a partial isometry $v\in M$ which satisfies $vv^*=p$ and $v^*v=q$.
We can identify $x \in (p+q)M(p+q) (\subset M)$ with 
$\begin{pmatrix}
pxp & pxv^*\\
vxp & vxv^*
\end{pmatrix}
\in M_2(pMp)$. 
Then, it follows 
\[
m(p,q) = \left\{ \frac{1}{2} 
\begin{pmatrix}
1 & u\\
u^* & 1
\end{pmatrix}
\middle|\,
u\in \U(pMp)
\right\} 
\subset (p+q)M(p+q) \subset M.
\]
Let $u\in \U(pMp)$ and put $\displaystyle e:= \frac{1}{2} 
\begin{pmatrix}
1 & u\\
u^* & 1
\end{pmatrix} \in m(p,q)$. 
Then the same discussion as in \cite[Lemma 2.5]{GS2} shows that, the only path $\gamma\colon [0, \pi/2] \to \P$ as in Condition is given by 
\[
\gamma(\theta) =  
\begin{pmatrix}
\cos^2\theta & u\cos\theta\sin\theta\\
u^*\cos\theta\sin\theta & \sin^2\theta
\end{pmatrix},\quad \theta \in [0, \pi/2].
\]\medskip

Suppose $p$ and $q$ satisfy Condition. 
We decompose $p$ and $q$ as above: 
\[
\begin{gathered}
\H = (p\wedge q^{\perp}) \H \oplus (p^{\perp} \wedge q) \H \oplus (p\wedge q) \H \oplus (p^{\perp}\wedge q^{\perp}) \H \oplus e_1\H \oplus e_2\H, \\ 
p= \begin{pmatrix}
1&0&0&0&0&0 \\
0&0&0&0&0&0 \\
0&0&1&0&0&0 \\
0&0&0&0&0&0 \\
0&0&0&0&1&0 \\
0&0&0&0&0&0
\end{pmatrix},\quad 
q= \begin{pmatrix}
0&0&0&0&0&0 \\
0&1&0&0&0&0 \\
0&0&1&0&0&0 \\
0&0&0&0&0&0 \\
0&0&0&0&a^2&ab \\
0&0&0&0&ba&b^2
\end{pmatrix}. 
\end{gathered}
\]
Since $m(p, q)\neq 0$, by the same discussion as in \cite[Lemma 2.4]{GS2}, there exists a partial isometry $v\in M$ which satisfies $vv^* = p\wedge q^{\perp}$ and $v^*v = p^{\perp} \wedge q$. 
If $p\wedge q^{\perp} = 0$, then the condition $\lVert p-q\rVert = 1$ implies $\lVert b\rVert = 1$. 
As \cite[Lemma 2.9]{GS2}, there exists a projection $p_0\in m(p, q)$ which admits more than one path with the property as in Condition. 
We can also show by \cite[Lemmas 2.8 and 2.9]{GS2} that the following does not happen: $p\wedge q^{\perp} \neq 0$ and $e_1\neq 0$. 
Thus we have $p\wedge q^{\perp} \neq 0$ and $0 = e_1\, (\sim e_2)$. 
Then $p$ and $q$ commutes. 
If there exist subprojections $0 \neq p_1\leq p\wedge q$ and $q_1\leq p^{\perp} \wedge q^{\perp}$ in $M$ which satisfy $p_1 \sim q_1$, 
then we can easily construct more than one path for the projection $p_0 = (p\wedge q^{\perp} + v + v^* +  p^{\perp} \wedge q)/2 + p\wedge q \in m(p, q)$, which contradicts Condition. 
Hence there exists a central projection $r\in M$ with $p\wedge q \leq r^{\perp}$ and $p^{\perp} \wedge q^{\perp} \leq r$. 
It follows $pr \perp qr$ and $p^{\perp}r^{\perp} \perp q^{\perp}r^{\perp}$. 
\end{proof}

We begin the proof of Theorem \ref{thm-G}. 
Let $\P\subset M$ and $\Q\subset N$ be proper Grassmann spaces and suppose $T\colon \P\to\Q$ is a surjective isometry. 
The preceding proposition implies that, for $p, q \in\P$, $p\vartriangle q$ if and only if $T(p)\vartriangle T(q)$. 
\smallskip

By the comparison theorem, there exists a central projection $r_0\in \P(M)$ which satisfies 
$pr_0 \prec p^{\perp}r_0$ and $pr_0^{\perp} \succ p^{\perp}r_0^{\perp}$ for some (and thus every) $p\in\P$. 
We say that a mapping between Grassmann spaces (or between von Neumann algebras) is \emph{typical} 
if it can be written as in the equation in the statement of Theorem \ref{thm-G}.
Since the composition of two typical mappings is also typical, in order to show Theorem \ref{thm-G}, we may and do assume that $p\prec p^{\perp}$ for every $p\in\P$ and $q\prec q^{\perp}$ for every $q\in\Q$.\smallskip

Our next task is to decompose $T$ into two mappings. 
We need preliminaries. 

\begin{lemma}
Let $\P\subset M$ be a proper Grassmann space in a von Neumann algebra $M\subset \B(\H)$ with $p\prec p^{\perp}$ for every $p\in \P$. 
\begin{enumerate}[$(a)$]
\item If $e, f \in \P$ and $e \vartriangle f$, then $m(e, f)$ is isometric to $\U(eMe)/2\,\, (=\{u/2 \mid u\in \U(eMe)\})$.  
\item Suppose $p_1, p_2\in \P$ satisfy $\lVert p_1-p_2\rVert < 1$. Then there exist projections $e, f\in\P$ such that $e\vartriangle f$ and $p_1, p_2\in m(e, f)$. 
\end{enumerate}
\end{lemma}
\begin{proof}
$(a)$ It suffices to consider two cases: $e\perp f$ or $e^{\perp}\perp f^{\perp}$. 
In the former case, there exists a partial isometry $v\in M$ such that $vv^*= e$ and $v^*v= f$. 
It follows $m(e, f) = \{(e + uv + v^*u^* + f)/2 \mid u\in \U(eMe)\}$, which is isometric to $\U(eMe)/2$. 
In the latter case, we similarly obtain that $m(e, f)$ is isometric to $\U(e^{\perp}Me^{\perp})/2$. 
In addition, we have $e \prec e^{\perp} \leq f \sim e$, thus $\U(e^{\perp}Me^{\perp})$ is isometric to $\U(eMe)$.\medskip

\noindent
$(b)$ By the generalization of Halmos' theorem, we can consider $p_1$ and $p_2$ as 
\[p_1= \begin{pmatrix}
1&0&0&0\\
0&0&0&0\\
0&0&1&0\\
0&0&0&0
\end{pmatrix},\quad 
p_2= \begin{pmatrix}
1&0&0&0\\
0&0&0&0\\
0&0&a^2&ab\\
0&0&ba&b^2
\end{pmatrix}
\]
through the decomposition $
\H =  (p\wedge q) \H \oplus (p^{\perp}\wedge q^{\perp}) \H \oplus e_1\H \oplus e_2\H$. 
By the comparison theorem, we may assume $p\wedge q \prec p^{\perp}\wedge q^{\perp}$ or $p^{\perp}\wedge q^{\perp} \prec p\wedge q$. 
In the former case, take a partial isometry $v\in M$ with $vv^* = p\wedge q$ and $v^*v \leq p^{\perp}\wedge q^{\perp}$. 
Put 
\[
e := \frac{1}{2}\begin{pmatrix}
1&v&0&0\\
v^*&v^*v&0&0\\
0&0&1&i\\
0&0&-i&1
\end{pmatrix}, \quad 
f := \frac{1}{2}\begin{pmatrix}
1&-v&0&0\\
-v^*&v^*v&0&0\\
0&0&1&-i\\
0&0&i&1
\end{pmatrix}. 
\] 
Then it is not difficult to see $e \perp f$ and $p_1, p_2\in m(e, f)$. 
The latter case can be proved similarly. 
\end{proof}

In addition, we recall Hatori and Moln\'ar's theorem. 
We remark that every Jordan $^*$-isomorphism between von Neumann algebras decomposes to the direct sum of a $^*$-isomorphism and a $^*$-antiisomorphism \cite[Exercise 10.5.26]{KR}.   
\begin{theorem}[Hatori and Moln\'ar, {\cite[Corollary 3]{HM}}]
Let $M$ and $N$ be von Neumann algebras. 
Suppose that $\tau\colon \U(M)\to \U(N)$ is a surjective isometry. 
Then there exist a central projection $e\in \P(N)$ and a Jordan $^*$-isomorphism $j\colon M\to N$ which satisfy 
$\tau(u) = \tau(1)(j(u)e + j(u)^*e^{\perp})$, $u\in \U(M)$. 
\end{theorem}

We return to the proof of Theorem \ref{thm-G}. 
There exists a unique central projection $r_1\in\P(M)$ which satisfies   
$pr_1$ is a finite projection and $pr_1^{\perp}$ is a properly infinite projection in $M$ for every $p\in\P$.
We define $\P_{\mathrm{fin}} := \{pr_1\mid p\in\P\}$ and $\P_{\mathrm{infin}} := \{pr_1^{\perp}\mid p\in\P\}$. 
Note that, if $r_1\neq0$ (resp.\ $r_1\neq 1$), $\P_{\mathrm{fin}}$ (resp.\ $\P_{\mathrm{infin}}$) is a proper Grassmann space in $Mr_1$ (resp.\ $Mr_1^{\perp}$)  and every projection in $\P_{\mathrm{fin}}$ (resp.\ $\P_{\mathrm{infin}}$) is a finite (resp.\ properly infinite) projection. 

\begin{lemma}
There exist surjective isometries 
$T_{\mathrm{fin}}\colon \P_{\mathrm{fin}} \to \Q_{\mathrm{fin}}$ and 
$T_{\mathrm{infin}}\colon \P_{\mathrm{infin}} \to \Q_{\mathrm{infin}}$
which are uniquely determined by the equation 
\[
T(p) = T_{\mathrm{fin}}(pr_1) + T_{\mathrm{infin}}(pr_1^{\perp}),\quad p\in\P. 
\]
\end{lemma}

\begin{proof}
Take the central projection $r_2\in\P(N)$ such that $\Q_{\mathrm{fin}} = \{qr_2\mid q\in\Q\}$ and $\Q_{\mathrm{infin}} = \{qr_2^{\perp}\mid q\in\Q\}$. 
Let $p_1, p_2 \in\P$. What we have to show are the following:
\begin{enumerate}[$(a)$]
\item If $p_1r_1 = p_2r_1$, then $T(p_1)r_2 = T(p_2)r_2$.
\item If $p_1r_1^{\perp} = p_2r_1^{\perp}$, then $T(p_1)r_2^{\perp} = T(p_2)r_2^{\perp}$.
\end{enumerate}
We show $(a)$ and $(b)$ at the same time. 
Since every Grassmann space is path-connected, it suffices to show them in the case $\lVert p_1-p_2\rVert < 1$. 
In this case, take projections $e, f$ as in the proof of the preceding lemma. 
It follows $e \vartriangle f$, $p_1, p_2 \in m(e, f)$ and thus $T(e) \vartriangle T(f)$, $T(p_1), T(p_2) \in m(T(e), T(f))$.  
Then $T$ restricts to a bijection from $m(e, f)$ onto $m(T(e), T(f))$. 
By $(a)$ of the preceding lemma, it determines a surjective isometry $T_1$ from $\U(eMe)$ onto $\U(T(e)NT(e))$. 
Then we can apply the theorem due to Hatori and Moln\'ar. 
By the fact that every Jordan $^*$-isomorphism between two von Neumann algebras preserves finite (properly infinite) projections, 
it follows that $T_1$ is decomposed to the direct sum of two surjective isometries
$T_2\colon \U(r_1eMe) \to \U(r_2T(e)NT(e))$ and $T_3\colon \U(r_1^{\perp}eMe) \to \U(r_2^{\perp}T(e)NT(e))$. 
Now it is easy to see that $(a)$ and $(b)$ hold. 
\end{proof}
 
We say $\P$ is \emph{finite} if every $p\in \P$ is finite, and $\P$ is \emph{properly infinite} if every $p\in \P$ is properly infinite.
By the preceding lemma, what we have to do is to prove Theorem \ref{thm-G} in the case both $\P$ and $\Q$ are finite, or both $\P$ and $\Q$ are properly infinite. 
\smallskip

First we consider the case $\P$ and $\Q$ are finite. 
Thus the setting is as follows: 
Let $\P\subset M$ and $\Q\subset N$ be finite proper Grassmann spaces. 
Assume $p\prec p^{\perp}$ for every $p\in\P$ and $q\prec q^{\perp}$ for every $q\in\Q$. 
Suppose that $T\colon \P\to\Q$ is a surjective isometry. 
 
A key to the proof is the following lemma. 

\begin{lemma}\label{lem}
In the above setting, suppose $p_1, p_2 \in \P$ are mutually orthogonal elements. 
By our assumption, we have $T(p_1) \perp T(p_2)$. 
Then, $T$ restricts to a bijection $T_0\colon \{p\in \P\mid p\leq p_1+p_2\} \to \{q\in \Q\mid q\leq T(p_1)+T(p_2)\}$. 
Moreover, $T_0$ extends to a typical mapping from $(p_1+p_2)M(p_1+p_2)$ onto $(T(p_1)+T(p_2))N(T(p_1)+T(p_2))$.
\end{lemma}
\begin{proof}
Since $p_1 \sim p_2$, using the way as before, we can identify: $p_1=
\begin{pmatrix}
1 & 0\\
0 & 0
\end{pmatrix}$, 
$p_2 = 
\begin{pmatrix}
0 & 0\\
0 & 1
\end{pmatrix}$ and
\[ 
\begin{split}
m(p_1, p_2) &= 
\left\{\frac{1}{2}
\begin{pmatrix}
1 & u\\
u^* & 1
\end{pmatrix}
\middle| u\in \U(p_1Mp_1)
\right\} \\
&\subset M_2(p_1Mp_1) = (p_1+p_2)M(p_1+p_2) \subset M.
\end{split}
\]

Similarly, we identify $(T(p_1)+T(p_2))N(T(p_1)+T(p_2))$ with $M_2(T(p_1)NT(p_1))$. 
We may assume 
\[
T\left(\frac{1}{2}\begin{pmatrix}
1 & 1\\
1 & 1
\end{pmatrix}\right)
= \frac{1}{2}\begin{pmatrix}
1 & 1\\
1 & 1
\end{pmatrix}.
\]
Consider the restriction of $T$ to $m(p_1, p_2)$ and define a surjective isometry $\tau\colon \U(p_1Mp_1) \to \U(T(p_1)NT(p_1))$ by 
\[
T\left(\frac{1}{2}
\begin{pmatrix}
1 & u\\
u^* & 1
\end{pmatrix}
\right) = 
\frac{1}{2}
\begin{pmatrix}
1 & \tau(u)\\
\tau(u)^* & 1
\end{pmatrix},\quad u\in \U(p_1Mp_1). 
\]
By the theorem due to Hatori and Moln\'ar, there exist central projections $r_1, r_2, r_3, r_4\in \P(p_1Mp_1)$ and $r'_1, r'_2, r'_3, r'_4 \in \P(T(p_1)NT(p_1))$ which satisfy
\[
r_1+r_2+r_3+r_4 = p_1,\quad r'_1+r'_2+r'_3+r'_4 = T(p_1)
\]
and a $^*$-isomorphism $\varphi_1\colon Mr_1 \to Nr'_1$, 
a $^*$-antiisomorphism $\varphi_2\colon Mr_2\to Nr'_2$,  
a conjugate linear $^*$-isomorphism $\varphi_3\colon Mr_3\to Nr'_3$, 
a conjugate linear $^*$-antiisomorphism $\varphi_4\colon Mr_4\to Nr'_4$ such that  
$\tau(u) = \varphi_1(ur_1) + \varphi_2(ur_2) +\varphi_3(ur_3) +\varphi_4(ur_4)$, $u\in  \U(p_1Mp_1)$. 
We define a typical mapping $\widetilde{T}$ from $M_2(p_1Mp_1)$ onto $M_2(T(p_1)NT(p_1))$ by 
\[
\begin{split}
\widetilde{T}
\begin{pmatrix}
x & y\\
z & w
\end{pmatrix}
:=
& \begin{pmatrix}
\varphi_1(xr_1) & \varphi_1(yr_1)\\
\varphi_1(zr_1) & \varphi_1(wr_1) 
\end{pmatrix} +
\begin{pmatrix}
r'_2 - \varphi_2(wr_2) & \varphi_2(yr_2)\\
\varphi_2(zr_2) & r'_2 - \varphi_2(xr_2) 
\end{pmatrix}\\
& +
\begin{pmatrix}
\varphi_3(xr_3) & \varphi_3(yr_3)\\
\varphi_3(zr_3) & \varphi_3(wr_3) 
\end{pmatrix}^* +
\begin{pmatrix}
r'_4 - \varphi_4(wr_4) & \varphi_4(yr_4)\\
\varphi_4(zr_4) & r'_4 - \varphi_4(xr_4) 
\end{pmatrix}^*, 
\end{split}
\]
$x, y, z, w\in p_1Mp_1$.
We show that this is an extension of $T_0$. 

Let $p$ be an element in $\P$ with the property $p\leq p_1+p_2$. 
By the finiteness of $\P$, there exist positive elements $a, b\in p_1Mp_1$ and a unitary $w\in \U(p_1Mp_1)$ with the property 
\[
a^2+b^2=p_1,\quad p= \begin{pmatrix}
a^2 & abw\\
w^*ba & w^*b^2w
\end{pmatrix}.
\]
Then $p$ is an element of 
\[
m:=
m\left( \frac{1}{2}\begin{pmatrix}
1 & iw\\
-iw^* & 1
\end{pmatrix}, \frac{1}{2}\begin{pmatrix}
1 & -iw\\
iw^* & 1
\end{pmatrix}\right), 
\]
so it follows $T(p)\leq T(p_1)+ T(p_2)$. 
We have to show that the mapping $\Phi$ from $\{p\in\P\mid p\leq p_1+p_2\}$ onto itself 
which is defined by $\Phi(p) = \widetilde{T}^{-1} \circ T (p)$ is the identity mapping. 
We already know that the projections 
\[
\begin{pmatrix}
1 & 0\\
0 & 0
\end{pmatrix},\quad
\begin{pmatrix}
0 & 0\\
0 & 1
\end{pmatrix}
\quad
\text{and}\quad
\frac{1}{2}\begin{pmatrix}
1 & u\\
u^* & 1
\end{pmatrix},\,\,
u\in\U(p_1Mp_1)
\]
are all fixed under $\Phi$.

It follows $\Phi$ restricts to a bijection from $m$ (as above) onto itself. 
It suffices to show that $\Phi$ restricts to the identity mapping on $m$.  
The self-adjoint unitary 
$U:= \displaystyle \frac{1}{\sqrt{2}}\begin{pmatrix}
1 & iw\\
-iw^* & -1
\end{pmatrix}$ 
gives rise to an isometry $\operatorname{Ad}(U)$ on $M_2(p_1Mp_1)$. 
Then $m$ is isometric to
\[
\begin{split}
\operatorname{Ad}(U) m &= \frac{1}{\sqrt{2}}\begin{pmatrix}
1 & iw\\
-iw^* & -1
\end{pmatrix}
m\left( \frac{1}{2}\begin{pmatrix}
1 & iw\\
-iw^* & 1
\end{pmatrix}, \frac{1}{2}\begin{pmatrix}
1 & -iw\\
iw^* & 1
\end{pmatrix}\right)
\frac{1}{\sqrt{2}}\begin{pmatrix}
1 & iw\\
-iw^* & -1
\end{pmatrix}\\
&= 
m\left(\begin{pmatrix}
1 & 0\\
0 & 0
\end{pmatrix}, \begin{pmatrix}
0 & 0\\
0 & 1
\end{pmatrix}\right)
= \left\{
\frac{1}{2}
\begin{pmatrix}
1 & v\\
v^* & 1
\end{pmatrix}
\middle| v\in \U(p_1Mp_1)
\right\}.
\end{split}
\]
Our task is to show that the mapping $\operatorname{Ad}(U) \circ \Phi \circ \operatorname{Ad}(U)$ is equal to the identity mapping on $\displaystyle \left\{
\frac{1}{2}
\begin{pmatrix}
1 & v\\
v^* & 1
\end{pmatrix}
\middle| v\in \U(p_1Mp_1)
\right\}$.
We have 
\[
\operatorname{Ad}(U)\begin{pmatrix}1&0\\ 0&0\end{pmatrix} = \frac{1}{2}\begin{pmatrix}1&iw\\ -iw^*&1\end{pmatrix}
\]
and
\[
\operatorname{Ad}(U)\left(
\frac{1}{2}\begin{pmatrix}
1 & u\\
u^* & 1
\end{pmatrix}\right) = 
\frac{1}{4}\begin{pmatrix}
2-iuw^*+iwu^* & -u-wu^*w\\
-u^*-w^*uw^* & 2-iu^*w+iw^*u
\end{pmatrix}
\]
for every $u\in\U(p_1Mp_1)$. 
In particular, for every self-adjoint unitary $a\in \U(p_1Mp_1)$, we have 
\[\operatorname{Ad}(U)\left(
\frac{1}{2}\begin{pmatrix}
1 & -aw\\
-w^*a & 1
\end{pmatrix}\right) = 
\frac{1}{2}
\begin{pmatrix}
1 & aw\\
w^*a & 1
\end{pmatrix}.
\] 
Therefore, if $v = iw$ or $v=aw$ for some self-adjoint unitary $a$, then 
\[\operatorname{Ad}(U) \circ \Phi \circ \operatorname{Ad}(U)\left(
\frac{1}{2}
\begin{pmatrix}
1 & v\\
v^* & 1
\end{pmatrix}\right) = \frac{1}{2}
\begin{pmatrix}
1 & v\\
v^* & 1
\end{pmatrix}.
\]
By the Hatori-Moln\'ar theorem, the same equation holds for every $v\in \U(p_1Mp_1)$.
\end{proof}

In fact, we may assume that the above typical mapping is always a Jordan $^*$-isomorphism. 
We explain this. 

First, take central projections $r_a, r_b, r_c\in\P(M)$ with $r_a+r_b+r_c = 1$ such that  
\begin{itemize}
\item $r_ap$ is an abelian projection for every $p\in\P$,
\item $r_b p\sim r_bp^{\perp}$ for every $p\in\P$, and  
\item $r_cpMp$ does not admit a type I$_1$ direct summand for every $p\in\P$, and $z(1-p_1-p_2)r_c = r_c$ for arbitrary $p_1, p_2\in\P$ with $p_1\perp p_2$. 
\end{itemize}
Fix $p_1, p_2\in\P$ with $p_1\perp p_2$. 
Since $r_ap$ is an abelian projection, we can take $\widetilde{T}$ as in the above proof so that it is a Jordan $^*$-homomorphism on $r_a (p_1+p_2)M (p_1+p_2)$. 
We show that $\widetilde{T}$ is also a Jordan $^*$-homomorphism on $r_c (p_1+p_2)M (p_1+p_2)$. 
By the condition of $r_c$, we can take a projection $e \in \P(M)$ such that $e\leq r_cp_2$, $r_c z(e) = r_c = r_c z(p_2-e)$ and $e \prec (1-p_1-p_2)$. 
Consider the restrictions of $T$ to the subset $S = \{p\in\P\mid p\leq p_1+ e\}$.  Note that $T$ is equal to $\widetilde{T}$ on this subset. 
Put $S_1 := \{p\in \P\mid p\perp (p_1 + e)\}$. 
It follows $S = \{p\in \P\mid p\perp S_1 \}$.  
Since $T$ preserves orthogonality, we have $\widetilde{T}(S) = T(S) = \{q\in \Q\mid q\perp T(S_1)\}$. 
If $\widetilde{T}$ is not a Jordan $^*$-homomorphism on $r_c (p_1+p_2)M (p_1+p_2)$, then $\widetilde{T}(S)$ cannot be written as above. 
Hence $\widetilde{T}$ is a Jordan $^*$-homomorphism on $r_c (p_1+p_2)M (p_1+p_2)$. 

Note that $r_b(p_1+p_2)=r_b$. 
We can take a typical mapping $\psi\colon r_bM\to r_bM$ with the property that $\widetilde{T}\circ \psi\colon r_bM\to N$ is a Jordan $^*$-homomorphism. 
Define the typical mapping $\Psi\colon M\to M$ by $\Psi(x):=\psi(r_bx)+(r_a+r_c)x$, $x\in M$. 
By the assumption concerning $r_b$, we have $\Psi(\P)=\P$.  
Considering the composition $T\circ \Psi$ instead of $T$, we may assume $\widetilde{T}$ is a Jordan $^*$-isomorphism. 
Let $p_3, p_4\in \P$ satisfy $p_3\perp p_4$. 
There exists $p_5\in \P$ such that $p_1\perp p_5$ and $p_3\leq p_1+p_5$. 
Note that $r_b(p_1+p_5) = r_b$. 
Considering the restriction of $T$ to the set $\{p\in \P\mid p\leq (r_a+r_c)p_1 + r_b\}$ and using the same discussion as above, we see that the restriction of $T$ to the subset $\{p\in\P\mid p\leq p_1+p_5\}$ extends to a Jordan $^*$-isomorphism from $(p_1+p_5)M(p_1+p_5)$ onto $(T(p_1)+T(p_5))N(T(p_1)+T(p_5))$. 
Since $p_3\leq p_1+p_5$, considering the restriction of $T$ to the set $\{p\in \P\mid p\leq (r_a+r_c)p_3 + r_b\}$, we also see that the restriction of $T$ to the subset $\{p\in\P\mid p\leq p_3+p_4\}$ extends to a Jordan $^*$-isomorphism from $(p_3+p_4)M(p_3+p_4)$ onto $(T(p_3)+T(p_4))N(T(p_3)+T(p_4))$.\medskip

Recall that a bijection $F$ from $\P(M)$ onto $\P(N)$ is called an \emph{orthoisomorphism} when it satisfies $pq = 0$ if and only if $F(p)F(q) = 0$, for $p, q\in \P(M)$.\medskip

We show that, under the above assumptions, the mapping $T$ extends uniquely to an orthoisomorphism from $\P(M)$ onto $\P(N)$. 

First, we extend $T$ to a mapping $T_1$ from 
$\{e\in\P(M)\mid e\leq p\,\, \text{for some}\,\, p\in\P\}$ to 
$\{f\in\P(N)\mid f\leq q\,\, \text{for some}\,\, q\in\Q\}$ by
\[
T_1(e) := \bigwedge\{T(p)\mid p\in\P, e\leq p\}. 
\]
We show that $T_1$ is a bijection which preserves orthogonality in both directions. 
Fix $e$. 
Take some $p_0\in\P$ with $e\leq p_0$ and $f\in \P(M)$ with $e\sim f\leq p_0^{\perp}$. 
We prove $T_1(e) = T(p_0) - T(p_0)T((p_0-e)+f)$. 
Suppose $p_1\in\P$ satisfies $e \leq p_1$. 
There exists a projection $p_2\in\P$ with the property 
$p_2\perp p_0$ and $f, p_1 \leq p_0+p_2$. 
Then $T$ restricts to a bijection $T_0\colon \{p\in \P\mid p\leq p_0+p_2\} \to \{q\in \Q\mid q\leq T(p_0)+T(p_2)\}$ and $T_0$ extends to a Jordan $^*$-isomorphism $J_0$ from $(p_0+p_2)M(p_0+p_2)$ onto $(T(p_0)+T(p_2))N(T(p_0)+T(p_2))$. 
Hence we obtain $T(p_1) = J_0(p_1) \geq J_0(e) = J_0(p_0)- J_0(p_0)J_0((p_0-e)+f) = T(p_0) - T(p_0)T((p_0-e)+f)$ for any $p_1\in\P$ with $e\leq p_1$ and thus $T_1(e) \geq T(p_0) - T(p_0)T((p_0-e)+f)$.  
In addition, we have $T(p_0) - T(p_0)T((p_0-e)+f) = J_0(e) = J_0(p_0)J_0((p_2-f)+ e) = T(p_0)T((p_2-f)+ e) \geq T_1(e)$. 
It follows $T_1(e) = T(p_0) - T(p_0)T((p_0-e)+f)$. 

Let $p_3, p_4\in\P$ be mutually commuting projections. 
Put $e= p_3p_4$, $p_0=p_3$, take some $f\in\P(M)$ so that $e\sim f \leq 1-p_3\vee p_4$ and put $p_2= (p_4-e)+f$. 
Then the above discussion shows that $T_1(p_3p_4) = T_1(e) = T(p_0)T((p_2-f)+e) = T(p_3)T(p_4)$. 
Thus $T_1$ is determined uniquely by the condition $T_1(p_3p_4)= T(p_3)T(p_4)$ for an arbitrary pair of mutually commuting projections $p_3, p_4\in\P$.   
It follows $T_1$ is a bijection with its inverse $T_1^{-1}\colon \{f\in\P(N)\mid f\leq q\,\, \text{for some}\,\, q\in\Q\} \to \{e\in\P(M)\mid e\leq p\,\, \text{for some}\,\, p\in\P\}$ which is defined by $T_1^{-1}(f) := \bigwedge\{q\mid q\in\Q, f\leq q\}$.
Since $T$ preserves orthogonality in both directions, so does $T_1$. 

We define a mapping $T_2\colon \P(M) \to \P(N)$ by 
\[
\begin{split}
T_2(p) &:= \bigvee \{T_1(e) \mid e\leq p,\,\, e\leq p_0 \text{ for some } p_0 \in \P\}\\
&= \bigwedge \{T_1(e)^{\perp} \mid e\perp p,\,\, e\leq p_0 \text{ for some } p_0 \in \P\}.
\end{split}
\]
It follows $T_2$ is an orthoisomorphism which extends $T$.\medskip 

Lastly, we rely on the following theorem due to Dye and its slight extension by the author. 
\begin{theorem}[Dye, {\cite[Corollary of Theorem 1]{D}} ($+$ {\cite[Proposition 5.2]{Mor}}) ]
Let $M$ and $N$ be two von Neumann algebras. 
Suppose $T\colon \P(M)\to \P(N)$ is an orthoisomorphism which preserves the distances between maximal abelian projections in the type I$_2$ direct summands.
Then there exists a Jordan $^*$-isomorphism from $M$ onto $N$ which extends $T$. 
\end{theorem}

Since our assumption shows that $T_2$ restricts to a surjective isometry between the classes of maximal abelian projections in the type I$_2$ direct summands, 
$T_2$ extends to a Jordan $^*$-isomorphism from $M$ onto $N$. 
This completes the proof of Theorem \ref{thm-G} when $\P$ and $\Q$ are finite.\medskip\medskip

Next we consider the case both $\P$ and $\Q$ are properly infinite. 
Thus the setting is as follows: 
Let $\P\subset M$ and $\Q\subset N$ be properly infinite proper Grassmann spaces. 
Assume $p\prec p^{\perp}$ for every $p\in\P$ and $q\prec q^{\perp}$ for every $q\in\Q$. 
Suppose that $T\colon \P\to\Q$ is a surjective isometry. 

The first step is to show that we may assume $T$ preserves orthogonality in both directions. 
As in \cite{GS2}, for two projections $p_1, p_2\in \P$, we write $p_1 \sharp p_2$ when $p_1 \perp p_2$ and $p_1 \prec (1-p_1-p_2)$. 
\smallskip

Since $\P$ is properly infinite, we can take mutually orthogonal projections $p_1, p_2, p_3\in\P$. 
We have $p_1 \vartriangle p_2$, $p_2 \vartriangle p_3$, $p_3 \vartriangle p_1$, thus $T(p_1) \vartriangle T(p_2)$, $T(p_2) \vartriangle T(p_3)$, $T(p_3) \vartriangle T(p_1)$.
It follows there exists a central projection $r\in \P(M)$ such that $T(p_1)r, T(p_2)r, T(p_3)r$ are mutually orthogonal and $T(p_1)^{\perp}r^{\perp}, T(p_2)^{\perp}r^{\perp}, T(p_3)^{\perp}r^{\perp}$ are mutually orthogonal. 
Composing $T$ with the typical mapping $q\mapsto qr+ q^{\perp}r^{\perp}$ on $\Q$, 
we may assume that $T(p_1), T(p_2), T(p_3)$ are mutually orthogonal.\smallskip 

Under this assumption, we show that, for any projections $p, p_0 \in \P$, we have $p\sharp p_0$ if and only if $T(p) \sharp T(p_0)$. 
Suppose $p\sharp p_0$. We have $p\sim p_1$, $p_0\sim p_2$. 
Since $\P$ is properly infinite, we obtain $(1-p-p_0) \sim ((1-p-p_0) + p+p_0) = 1$ and similarly $(1- p_1-p_2) \sim 1$, thus $(1-p-p_0) \sim (1-p_1-p_2)$. 
Therefore there exists a unitary $u \in \U(M)$ which satisfies $upu^* = p_1$ and $up_0u^* = p_2$. 
By the functional calculus on $M$, there exists a self-adjoint operator $a\in M_{sa}$ with $u = e^{ia}$. 
We show $T(e^{ita} p e^{-ita}) \sharp T(e^{ita}p_0e^{-ita})$ for every $t\in[0, 1]$. 
It suffices to show $T(p) \sharp T(p_0)$ when $\lVert p-p_1\rVert < 1/2$ and $\lVert p_0-p_2\rVert < 1/2$. 
In that case, we have 
\[
\begin{split}
& \lVert (1-T(p)-T(p_0)) - (1-T(p_1) - T(p_2)) \rVert \\
\leq &\,  \lVert  T(p_1) - T(p) \rVert + \lVert T(p_2)-T(p_0) \rVert = 
\lVert  p_1 - p \rVert + \lVert p_2-p_0 \rVert < 1.
\end{split}
\]
Combine this inequality with $T(p) \vartriangle T(p_0)$ to obtain $T(p) \perp T(p_0)$. 
Moreover, we can apply the generalization of Halmos' theorem to the two projections $1-T(p)-T(p_0)$ and $1-T(p_1) - T(p_2)$ to obtain $(1-T(p)-T(p_0)) \sim (1-T(p_1)-T(p_2))$. 
Thus we have $T(p) \sharp T(p_1)$. \smallskip

We have shown that $T$ preserves the relation $\sharp$ in both directions.  
It is easy to see that for $p_1, p_2\in \P$, we have $p_1\leq p_2$ if and only if $\{p\in\P \mid p\sharp p_1\} \supset \{p\in\P \mid p\sharp p_2\}$. 
Thus we obtain $p_1\leq p_2$ if and only if $T(p_1) \leq T(p_2)$. 

Let $p_1, p_2\in \P$ satisfy $p_1 \vee p_2 \in \P$. 
Since $p_1 \vee p_2$ is the minimum projection in $\P$ which majorizes both $p_1$ and $p_2$, we have $T(p_1 \vee p_2) = T(p_1) \vee T(p_2)$. 
Similarly, if $p_1, p_2\in \P$ satisfy $p_1 \wedge p_2 \in \P$, then $T(p_1 \wedge p_2) = T(p_1) \wedge T(p_2)$. 

Let $p_1, p_2\in \P$ satisfy $p_1 \perp p_2$. 
Since $\P$ is properly infinite, there exist mutually orthogonal subprojections $p_{11}, p_{12} \in \P$ of $p_1$ which satisfy $p_1 = p_{11}+ p_{12}$. 
Since $p_{11} \sharp p_2$ and $p_{12} \sharp p_2$, we have $T(p_{11}) \sharp T(p_2)$ and $T(p_{12}) \sharp T(p_2)$. 
Hence we obtain $T(p_2) \perp (T(p_{11}) \vee T(p_{12})) = T(p_{11}\vee p_{12}) = T(p_1)$. 
Therefore, $T$ preserves orthogonality in both directions. 

We show a version of Lemma \ref{lem}.

\begin{lemma} 
Under the above assumptions, suppose $p_1, p_2 \in \P$ are mutually orthogonal. 
Then, $T$ restricts to a bijection $T_0\colon \{p\in \P\mid p\leq p_1+p_2\} \to \{q\in \Q\mid q\leq T(p_1)+T(p_2)\}$. 
Moreover, $T_0$ extends (uniquely) to a Jordan $^*$-isomorphism from $(p_1+p_2)M(p_1+p_2)$ onto $(T(p_1)+T(p_2))N(T(p_1)+T(p_2))$.
\end{lemma}
\begin{proof}
Using the same notations and discussions as in the proof of Lemma \ref{lem}, we can construct a typical mapping $\widetilde{T}$ from $(p_1+p_2)M(p_1+p_2)$ onto $(T(p_1)+T(p_2))N(T(p_1)+T(p_2))$. 
Take projections $p, \tilde{p}_1, \tilde{p}_2\in\P$ such that $p\leq p_1$, $p\sim (p_1-p)$ and $\tilde{p}_1 \leq \tilde{p}_2\leq p_2$, $\tilde{p}_1 \sim (\tilde{p}_2-\tilde{p}_1) \sim (p_2-\tilde{p}_2)$. 
By the same discussion as in Lemma \ref{lem}, we see that 
$T(p+\tilde{p}_1) = \widetilde{T}(p+\tilde{p}_1)$ and $T(p+\tilde{p}_2) = \widetilde{T}(p+\tilde{p}_2)$. 
It follows $ \widetilde{T}(p+\tilde{p}_1) \leq \widetilde{T}(p+\tilde{p}_2)$, which shows that $\widetilde{T}$ is actually a Jordan $^*$-isomorphism. 

We show $T(p) = \widetilde{T}(p)$ for every $p\in \P$ with $p\leq p_1+p_2$. 
Since $p\sim p_1 = \begin{pmatrix}1&0\\0&0\end{pmatrix}$, there exist $x, y\in p_1Mp_1$ which satisfy 
\[
x^*x+y^*y =p_1,\quad p = \begin{pmatrix} xx^*&xy^*\\ yx^*&yy^*\end{pmatrix}. 
\]
Let $x= v\lvert x\rvert$, $y= w\lvert y\rvert$ be polar decompositions. 
By the spectral theorem, we may assume that the spectral set $\sigma(\lvert x\rvert)$ of $\lvert x\rvert$ is a finite set. 
Thus $\lvert x\rvert = \sum_{k=1}^n \lambda_k e_k$ for some $0= \lambda_1< \lambda_2 <\cdots < \lambda_n = 1$ and mutually orthogonal projections $e_k \in \P(p_1Mp_1)$ such that $\sum_{k=1}^ne_k = p_1$. (Projections $e_1$ and $e_n$ may be 0.)  
We have $\lvert y\rvert = \sum_{k=1}^n \sqrt{1-\lambda_k^2} e_k$. 
Since $p_1$ is properly infinite, there exist subprojections $f_k\leq e_k$ in $\P(p_1Mp_1)$, $k=1, \ldots, n$, which satisfy the following property:  
$\sum_{k=1}^n f_k \sim \sum_{k=1}^n (e_k-f_k) \sim p_1$, and 
partial isometries $v\sum_{k=2}^n f_k$, $w\sum_{k=1}^{n-1} f_k$, $v\sum_{k=2}^n (e_k-f_k)$ and $w\sum_{k=1}^{n-1} (e_k-f_k)$ admit unitary extensions $v_0, w_0, v_1$ and $w_1\in \U(p_1Mp_1)$, respectively. 
We show that the projection
\[
\begin{split}
p_0& :=\begin{pmatrix}
v(\sum_{k=1}^n \lambda_k f_k)^2v^* & v(\sum_{k=1}^n \lambda_k f_k) (\sum_{k=1}^n \sqrt{1-\lambda_k^2} f_k) w^*\\
w(\sum_{k=1}^n \sqrt{1-\lambda_k^2} f_k) (\sum_{k=1}^n \lambda_k f_k)v^* & w(\sum_{k=1}^n \sqrt{1-\lambda_k^2} f_k)^2 w^*
\end{pmatrix} \\
&= \begin{pmatrix}
v_0(\sum_{k=1}^n \lambda_k f_k)^2v_0^* & v_0(\sum_{k=1}^n \lambda_k f_k) (\sum_{k=1}^n \sqrt{1-\lambda_k^2} f_k) w_0^*\\
w_0(\sum_{k=1}^n \sqrt{1-\lambda_k^2} f_k) (\sum_{k=1}^n \lambda_k f_k)v_0^* & w_0(\sum_{k=1}^n \sqrt{1-\lambda_k^2} f_k)^2 w_0^*
\end{pmatrix} 
\end{split}
\]
in $\P$ satisfies $T(p_0)=\widetilde{T}(p_0)$. 
Consider the projection
\[
\begin{split}
& \, p_0 + \begin{pmatrix}v_0 (\sum_{k=1}^n (e_k-f_k))v_0^*&0\\ 0&0\end{pmatrix}\\
=\, & \begin{pmatrix}
v_0((\sum_{k=1}^n \lambda_k f_k)^2 + \sum_{k=1}^n (e_k-f_k))v_0^* & v_0(\sum_{k=1}^n \lambda_k f_k) (\sum_{k=1}^n \sqrt{1-\lambda_k^2} f_k) w_0^*\\
w_0(\sum_{k=1}^n \sqrt{1-\lambda_k^2} f_k) (\sum_{k=1}^n \lambda_k f_k)v_0^* & w_0(\sum_{k=1}^n \sqrt{1-\lambda_k^2} f_k)^2 w_0^*
\end{pmatrix}\\
=\, & 
\begin{pmatrix}
a^2& abv_0w_0^*\\
w_0v_0^* ba & w_0v_0^*b^2v_0w_0^* 
\end{pmatrix},
\end{split}
\]
where $a:= v_0(\sum_{k=1}^n \lambda_k f_k + \sum_{k=1}^n (e_k-f_k))v_0^*$ and 
$b:= v_0 (\sum_{k=1}^n \sqrt{1-\lambda_k^2} f_k)v_0^*$. 
It follows $a, b\geq 0$, $a^2+b^2 = p_0$. 
By the same discussion as in Lemma \ref{lem}, we obtain 
\[
T\left(p_0 + \begin{pmatrix}v_0 (\sum_{k=1}^n (e_k-f_k))v_0^*&0\\ 0&0\end{pmatrix}\right) = \widetilde{T}\left(p_0 + \begin{pmatrix}v_0 (\sum_{k=1}^n (e_k-f_k))v_0^*&0\\ 0&0\end{pmatrix}\right)
\]
Similarly, we obtain
\[
T\left(p_0 + \begin{pmatrix}0&0\\ 0&w_0(\sum_{k=1}^n (e_k-f_k))w_0^*\end{pmatrix}\right) = \widetilde{T}\left(p_0 + \begin{pmatrix}0&0\\ 0&w_0(\sum_{k=1}^n (e_k-f_k))w_0^*\end{pmatrix}\right). 
\]
Since 
\[
\left(p_0 + \begin{pmatrix}v_0 (\sum_{k=1}^n (e_k-f_k))v_0^*&0\\ 0&0\end{pmatrix}\right) \wedge
\left(p_0 + \begin{pmatrix}0&0\\ 0&w_0(\sum_{k=1}^n (e_k-f_k))w_0^*\end{pmatrix}\right) = p_0, 
\]
we have $T(p_0)= \widetilde{T}(p_0)$. 
Similarly, we have $T(p-p_0)= \widetilde{T}(p-p_0)$. 
Finally, we have $T(p) = T(p_0)\vee T(p-p_0) = \widetilde{T}(p_0)\vee \widetilde{T}(p-p_0) = \widetilde{T}(p)$.
\end{proof}

A discussion which is similar to (or simpler than) that in finite cases shows that it is possible to extend $T$ to an orthoisomorphism from $\P(M)$ onto $\P(Q)$. 
By Dye's theorem, $T$ extends to a Jordan $^*$-isomorphism from $M$ onto $N$.$\hfill\Box$

\section{Isometries between projection lattices}\label{lattice}
In this section, we write $M\cong N$ when two von Neumann algebras $M$ and $N$ are Jordan $^*$-isomorphic.
\begin{theorem}\label{thm-l}
Let $M, N$ be von Neumann algebras without type I$_1$ direct summands. 
Then $M$ and $N$ are Jordan $^*$-isomorphic if and only if $\P(M)$ and $\P(N)$ are isometric. 
\end{theorem}

Suppose $T\colon \P(M) \to \P(N)$ is a surjective isometry. 
Since $M$ does not admit a type I$_1$ direct summand, there exists a projection $p\in \P(M)$ which satisfies $z(p)=z(p^{\perp})=1$. 
Take the (proper) Grassmann space $\P$ in $M$ which contains $p$. 
Then $T(\P)$ is a proper Grassmann space in $Nz(T(p)) z(T(p)^{\perp})$. 
By Theorem \ref{thm-G}, it follows that $M$ is Jordan $^*$-isomorphic to $Nz(T(p)) z(T(p)^{\perp})$, which is a direct summand of $N$. 
Similarly, $N$ is Jordan $^*$-isomorphic to a direct summand of $M$. 
Therefore, it suffices to show the following lemma.

\begin{lemma}
Let $M, N$ be von Neumann algebras. 
Suppose that $M$ is Jordan $^*$-isomorphic to a direct summand of $N$, and 
$N$ is Jordan $^*$-isomorphic to a direct summand of $M$. 
Then $M$ is Jordan $^*$-isomorphic to $N$.
\end{lemma}
\begin{proof}
There exist central projections $p\in\P(M)$ and $q\in\P(N)$ such that 
$M$, $N$ are Jordan $^*$-isomorphic to $Nq, Mp$, respectively. 
It follows 
\[
M = Mp \oplus Mp^{\perp} \cong N \oplus Mp^{\perp} = 
Nq \oplus Nq^{\perp} \oplus Mp^{\perp} \cong M \oplus Nq^{\perp} \oplus Mp^{\perp}.
\]
Take a Jordan $^*$-isomorphism $\Phi\colon M \oplus Nq^{\perp} \oplus Mp^{\perp} \to M $. 
We define $i\colon M \to M \oplus Nq^{\perp} \oplus Mp^{\perp}$ by 
$i(x) := x \oplus 0 \oplus 0$, $x\in M$.
Put $p_0:= \Phi(0 \oplus q^{\perp} \oplus p^{\perp})$ and $p_n := (\Phi\circ i)^n(p_0)$, $n\in \N$. 
Then $\{p_n\}_{n \geq 0}$ is an orthogonal family of central projections in $M$ and 
$Mp_n \cong Nq^{\perp} \oplus Mp^{\perp} \cong Mp_0$, $n\geq 0$.
Put $p_{\infty} := \vee_{n\geq 0} \, p_n$. 
We have 
\[
\begin{split}
M &= Mp_{\infty}^{\perp} \oplus Mp_{\infty}\\
&\cong Mp_{\infty}^{\perp} \oplus Mp_0 \mathbin{\overline{\otimes}} \ell^{\infty}\\ 
&\cong Mp_{\infty}^{\perp} \oplus Mp_0 \mathbin{\overline{\otimes}} \ell^{\infty} \oplus Mp_0 \mathbin{\overline{\otimes}} \ell^{\infty}\\
&\cong M \oplus Mp_0 \mathbin{\overline{\otimes}} \ell^{\infty}.
\end{split}
\]
Similarly, we obtain $N \cong N \oplus Mp_0 \mathbin{\overline{\otimes}} \ell^{\infty}$. 
Lastly, we have
\[
\begin{split}
M \oplus Mp_0 \mathbin{\overline{\otimes}} \ell^{\infty} &\cong 
Nq \oplus (Nq^{\perp} \oplus Mp^{\perp}) \mathbin{\overline{\otimes}} \ell^{\infty}\\
&\cong Nq \oplus Nq^{\perp} \oplus (Nq^{\perp} \oplus Mp^{\perp}) \mathbin{\overline{\otimes}} \ell^{\infty}\\
&= N \oplus Mp_0 \mathbin{\overline{\otimes}} \ell^{\infty}.
\end{split}
\]
\end{proof}

If in the above theorem we drop the condition concerning type I$_1$ summand, then we can find a counterexample. 
Indeed, any bijection between $\P(L^{\infty}([0, 1]))$ and $\P(L^{\infty}([0, 1]) \oplus \C)$ is isometric, but $L^{\infty}([0, 1])$ and $L^{\infty}([0, 1]) \oplus \C$ are not isomorphic.\medskip 

Theorem \ref{thm-G} also gives complete descriptions of surjective isometries between projection lattices of two von Neumann algebras. 
However, to give such a description in concrete situations is a complicated work. 
In the rest of this paper, we consider factor cases. 

Let $M$, $N$ be countably decomposable factors and 
suppose $T\colon \P(M) \to \P(N)$ is a surjective isometry.
Then Theorem \ref{thm-G} implies that $M$ and $N$ are Jordan $^*$-isomorphic, 
and thus $M$ and $N$ are $^*$-isomorphic or $^*$-antiisomorphic. 
We assume $M=N$. 
Note that only two points $0$ and $1$ are isolated in $\P(M)$, 
and thus $T$ restricts to a bijection on $\{0,1\}$.\smallskip

First we consider type I factors. 
Let $\H$ be a separable complex Hilbert space.
For $n\in \N = \{1, 2, \ldots\}$, the symbol $\P_n(\H)$ denotes the collection of rank $n$ projections in $\B(\H)$, and we put $\P^n(\H):=\{p^{\perp}\mid p\in\P_n(\H)\}$. 
The symbol $\P_{\infty}(\H)$ denotes the set of projections in $\B(\H)$ whose range and kernel are both infinite dimensional.

\begin{example}
If $M=\B(\H)$ is a type I$_N$ factor with $N \in \N$, 
then Grassmann spaces of $M$ are $\P_n(\H)$, $n = 1, 2, \ldots, N-1$.
In this case, there exists a mapping $\sigma$ from $\{1, 2, \ldots, N-1\}$ to $\{1, -1\}$ 
which satisfies the following conditions:
\begin{itemize} 
\item For $n=1,\ldots, N-1$, $\sigma(n) \sigma(N-n)=1$. 
\item If $\sigma(n) = 1$, the mapping $T$ restricts to a bijection $T_n$ from $\P_n(\H)$ onto itself. Moreover, $T_n$ extends uniquely to a $^*$-automorphism or a $^*$-antiautomorphism on $\B(\H)$. 
\item If $\sigma(n) = -1$, the mapping $T$ restricts to a bijection $T_n$ from $\P_n(\H)$ onto $\P_{N-n}(\H)$. Moreover, the mapping $p\mapsto 1-T_n(p)$, $p\in\P_n(\H)$ extends uniquely to a $^*$-automorphism or a $^*$-antiautomorphism on $\B(\H)$. 
\end{itemize}
\end{example}

\begin{example}
If $M=\B(\H)$ is a type I$_{\infty}$ factor, 
then  Grassmann spaces of $M$ are $\P_n(\H)$, 
$\P^n(\H)$, $n \in \N$ and $\P_{\infty}(\H)$. 
In this case, $T$ restricts to a bijection $T_{\infty}$ from $\P_{\infty}(\H)$ onto itself. 
Thus $T_{\infty}$ extends uniquely to a $^*$-automorphism or a $^*$-antiautomorphism, 
or the mapping $p\mapsto 1- T_{\infty}(p)$, $p\in \P$ extends uniquely to a $^*$-automorphism or a $^*$-antiautomorphism on $\B(\H)$. 
In addition, there exists a unique mapping $\sigma$ from $\N$ to $\{1, -1\}$ 
which satisfies the following conditions: 
\begin{itemize}
\item If $\sigma(n) = 1$, the mapping $T$ restricts to a bijection $T_n$ from $\P_n(\H)$ onto itself, and $T$ also restricts to a bijection $T^n$ from $\P^n(\H)$ onto itself. 
Each mapping extends uniquely to a $^*$-automorphism or a $^*$-antiautomorphism on $\B(\H)$. 
\item If $\sigma(n) = -1$, the mapping $T$ restricts to a bijection $T_n$ from $\P_n(\H)$ onto $\P^n(\H)$, and $T$ also restricts to a bijection $T^n$ from $\P^n(\H)$ onto $\P_n(\H)$. Thus the mappings $1-T_n$ and $1-T^n$ extends to a $^*$-automorphism or a $^*$-antiautomorphism on $\B(\H)$. 
\end{itemize}
\end{example}

Note that, for every $^*$-automorphism (resp.\ $^*$-antiautomorphism) $\Phi$ on $\B(\H)$, there exists a unitary (resp.\ antiunitary) $u$ on $\H$ which satisfies $\Phi(x)= uxu^*$ (resp.\ $\Phi(x) = ux^*u^*$), $x\in\B(\H)$.
Thus we see that our result actually generalizes the theorem due to Geh\'er and \v{S}emrl \cite[Theorem 1.2]{GS2}.

\begin{example}
If $M$ is a type II$_1$ factor with a normal tracial state $\tau$, 
then  Grassmann spaces of $M$ are 
$\P_{\lambda}(M) := \{p\in\P(M)\mid \tau(p)=\lambda\}$, $0<\lambda<1$. 
In this case, we can use the fact that every Jordan $^*$-automorphism on tracial factor preserves the trace. 
It follows there exists a unique mapping $\sigma\colon (0,1)\to \{1, -1\}$ which satisfies the following conditions: 
\begin{itemize} 
\item For $\lambda \in (0, 1)$, $\sigma(\lambda) \sigma(1-\lambda)=1$. 
\item If $\sigma(\lambda) = 1$, the mapping $T$ restricts to a bijection $T_{\lambda}$ from $\P_{\lambda}(M)$ onto itself. 
Moreover, $T_{\lambda}$ extends uniquely to a $^*$-automorphism or a $^*$-antiautomorphism on $M$. 
\item If $\sigma(\lambda) = -1$, the mapping $T$ restricts to a bijection $T_{\lambda}$ from $\P_{\lambda}(M)$ onto $\P_{1-\lambda}(M)$. Moreover, the mapping $p\mapsto 1-T_{\lambda}(p)$, $p\in\P_{\lambda}(M)$ extends uniquely to a $^*$-automorphism or a $^*$-antiautomorphism on $M$. 
\end{itemize}
\end{example}

\begin{example}
If $M$ is a type II$_{\infty}$ factor with a normal semifinite faithful tracial weight $\tau$, 
then  Grassmann spaces of $M$ are 
$\P_{(\lambda, 1)} := \{p\in\P(M)\mid \tau(p)=\lambda\}$, 
$\P_{(\lambda, -1)} := \{p^{\perp} \mid p\in \P_{\lambda}(M, \tau)\}$, $0<\lambda<\infty$, and 
$\P_{\infty} = \{p\in\P(M)\mid \tau(p)=\infty = \tau(p^{\perp})\}$.

This case is the most complicated. 
First, $T$ restricts to a bijection $T_{\infty}$ from $\P_{\infty}$ onto itself, 
and $T_{\infty}$ or $1-T_{\infty}$ extends to a $^*$-automorphism or a $^*$-antiautomorphism on $M$. 
In order to consider the other Grassmann spaces, we need to take the following multiplicative group into account: 
\[
\F := \{\lambda \in (0, \infty) \mid pMp\cong qMq\,\, \text{for some}\,\, p\in\P_{(1, 1)},\,\, q\in\P_{(\lambda, 1)} \}
\]
(Note that the symbol $\cong$ means that two algebras are Jordan $^*$-isomorphic. 
cf.\ The fundamental group of the II$_1$ factor $pMp$ is a subgroup of $\F$.)

There exists a bijection $f$ from $(0, \infty)\times \{1, -1\}$ onto itself which satisfies the following condition: 
Let $(\lambda, s), (\mu, t) \in(0, \infty)\times \{1, -1\}$ satisfy $f (\lambda, s) = (\mu, t)$. Then
\begin{itemize}
\item $\lambda/\mu\in \F$.  
\item The mapping $T$ restricts to a bijection $T_{(\lambda, s)}$ from $\P_{(\lambda, s)}$ onto $\P_{(\mu, t)}$. 
\item If $st=1$, then $T_{(\lambda, s)}$ extends uniquely to a $^*$-automorphism or a $^*$-antiautomorphism on $M$. 
\item If $st = -1$, the mapping $p\mapsto 1-T_{(\lambda, s)}(p)$, $p\in\P_{(\lambda, s)}$ extends uniquely to a $^*$-automorphism or a $^*$-antiautomorphism on $M$. 
\end{itemize}
\end{example}

\begin{example}
If $M$ is a type III factor, 
then the unique Grassmann space of $M$ is $\P:=\P(M)\setminus\{0, 1\}$. 
It follows that the restriction $T_0$ of $T$ on $\P$ is described as one and only one of the following four options: 
it extends uniquely to a $^*$-automorphism or a $^*$-antiautomorphism, or 
the mapping $p\mapsto 1-T_0(p)$, $p\in \P$ extends uniquely to a $^*$-automorphism or a $^*$-antiautomorphism. 
\end{example}

\medskip\medskip

\textbf{Acknowledgements} \quad 
The author appreciates Yasuyuki Kawahigashi who is the advisor of the author. 
Part of this research was performed while the author was visiting the Institute for Pure and Applied Mathematics (IPAM), which is supported by the National Science Foundation. 
This work was supported by Leading Graduate Course for Frontiers of Mathematical Sciences and Physics, MEXT, Japan.

\end{document}